\documentclass{amsart}
\newtheorem{thm}[subsection]{Theorem}
\newtheorem{defn}[subsection]{Definition}
\newtheorem{lemma}[subsection]{Lemma}
\newtheorem{cor}[subsection]{Corollary}
\newtheorem{prop}[subsection]{Proposition}
\theoremstyle{definition}

\newtheorem{rmk}[subsection]{Remark}

\numberwithin{equation}{section}

\begin{document}

\title[Perturbation of Closed Range Operators and Moore-Penrose inverse]{Perturbation of Closed Range Operators and Moore-Penrose inverse }
\author{S. H. Kulkarni}

\address{Department of Mathematics\\I. I. T. Madras, \\ Taramani, Chennai \\ India-600 036.}

\email{shk@iitm.ac.in}

\author{G. Ramesh}

\address{Office 305, Block E \\I. I. T. Hyderabad, Kandi (V)\\Sangareddy (M), Telangana \\ India-502 205.}

\email{rameshg@iith.ac.in}

\thanks{}
\subjclass[2000]{ 47A55 }

\date{\today.}

\keywords{closed operator, reduced minimum modulus, Moore-Penrose inverse}

\begin{abstract}
Let $H_1,H_2$ be complex Hilbert spaces and $T:H_1\rightarrow H_2$ be a densely defined closed operator with domain $D(T)\subseteq H_1$ and
$T^{\dagger}$ be the Moore-Penrose inverse of $T$. Let $S:H_1\rightarrow H_2$ be a bounded operator. In this article we focus our attention on the following questions:
\begin{itemize}
\item Under what conditions closedness of of range of $T$ will imply the closedness of range of $T+S$?
\item What is the relation between $T^{\dagger}$ and $(T+S)^{\dagger}$?
\item What is the relation between $T^{\dagger}$ and $S^{\dagger}$?
\end{itemize}

\end{abstract}
\maketitle
\section{Introduction}
Most of the  problems that we encounter in many branches of Mathematics (e.g. Partial Differential Equations, Optimization Theory, Numerical Analysis etc,) and several branches in Engineering and physics end up in  solving operator equations of the form
\begin{equation}\label{operatorequation}
Tx=y,
\end{equation}
where $y$ is a given vector in a Hilbert space, $T$ is a linear operator (possibly unbounded). We need to find out the vector $x$. In case, the operator $T$ is not invertible, the equation cannot have a solution. In such cases one alternate approach is to look for an approximate solution instead of the original solution. This is given by $x=T^{\dagger}y$, where $T^{\dagger}$ is the Moore-Penrose inverse of $T$, provided $y$ is in the domain of $T^{\dagger}$.  In  this case the equation (\ref{operatorequation}) is solvable for all $y$ in the domain of $T^{\dagger}$ and the solution is unique.

Another important task is the stability of the solution. The solution (least square solution of minimal norm) is stable only if $T^{-1}$ (or $T^{\dagger}$) is continuous. The perturbation theory tells us that how much we can perturb the operator without loosing the desired properties of the solution.

In this article we discuss the perturbation results for the Moore-Penrose inverse of a densely defined closed operator under two different set of conditions.

Let $H_1,H_2$ be complex Hilbert spaces and $T:H_1\rightarrow H_2$ be a  closed operator with dense domain $D(T)\subseteq H_1$. We prove perturbations results related to the perturbation of closed range of operators and Moore-Penrose inverses under different conditions. Here we state the results.

Let $T$ be a densely defined  closed operator.
\begin{itemize}
\item[(I)] \label{1sttype}
Suppose  $S\in \mathcal B(H_1,H_2)$ satisfies the
following conditions:

\begin{enumerate}\label{newdagger}
\item $||T^\dagger S||<1$
 \item $TT^\dagger S=S$  and
 \item $ST^\dagger T=S|_{D(T)}$.
 \end{enumerate}
 Then \begin{itemize}
\item[(a)] $R(T+S)$ is closed and
\item[(b)]$(T+S)^\dagger=(I+T^\dagger S)^{-1}T^\dagger$.
 \end{itemize}
\item[(II)]\label{2ndtype}
Let $S\in \mathcal B(H_1,H_2)$ be such that
\begin{equation*}
\|Sx\|\leq \lambda_1\|Tx\|+\lambda_2 \|Sx+Tx\|,\; \text{for all}\; x\in D(T),\; \text{where}\; \lambda_1 <1,\; i=1,2.
\end{equation*}
We prove the following results:
\begin{enumerate}
\item If $R(T)$ is closed, then $R(T+S)$ is closed
\item In addition if $T$ is surjective, then
\begin{itemize}
\item[(i)] $(T+S)^{\dagger}=T^{\dagger}(I+T^{\dagger}S)^{-1}$
\item[(ii)] $S^{\dagger}=T^{\dagger}\displaystyle \sum_{n=0}^\infty [(S-T)T^{\dagger}]^n=T^{\dagger}(I+(S-T)T^{\dagger})^{-1}$.
\end{itemize}
In particular, if $\lambda_2=0$, we get the following error bound:
\begin{equation*}
\|(S+T)^{\dagger}-T^{\dagger}\|\leq \dfrac{\|T^{\dagger}\|^2\|S\|}{1-\|ST^{\dagger}\|}.
\end{equation*}
\end{enumerate}
\end{itemize}
The  perturbation results in $(I)$ are the generalizations of the results in
\cite{stewart} from the case of rectangular matrices to the case of
closed operators on Hilbert spaces. Many authors have studies
 these results for Moore-Penrose inverses of bounded operators on
Hilbert spaces \cite{ben66,dinghuang94,dinghuang96,dingj} and Banach
spaces \cite{nashed}.

  The statements $(1)$ and $2(ii)$ of $(II)$, generalizes  the results of Jiu Ding \cite{jiuding} in one direction, namely from  bounded operators on Hilbert spaces to unbounded closed operators between Hilbert spaces. In fact Ding proved these statements for bounded operators on Banach spaces. Using the concept of gap between two subspaces Ding and Huang \cite{dinghuang94} proved some perturbation results for bounded operators on Hilbert spaces. These results were used to obtain perturbation results of frames in \cite{christensenclosedrange}. In the case of matrices the perturbation results were extensively studied by Stewart and Ben-Israel in \cite{ben66,stewart} and for Banach space operators by Nashed and Moore\cite{nashedmoore}. Some of these results also can be found in \cite{campbellmeyer}. We have noticed that these results were not studied for unbounded operators. In this article we attempt to extend the results of Jiu Ding \cite{jiuding} to unbounded operators. We also prove some results which are not available in \cite{jiuding}.
  In the process we prove the reverse order law for the Moore-Penrose inverses of closed operators under some assumptions.
  Some perturbation results for relative bounded operators are appeared in a recent article \cite{huangetall12}. The authors in \cite{huangetalbanach} discussed perturbation of closed range operators and Moore-Penrose inverses of relative bounded operators between Banach spaces with an extra assumption that is similar to the one we considered. The article \cite{shkgrcgt} contains some perturbation results for Moore-Penrose inverses of closed operators between Hilbert spaces with respect to a new topology.

We organize the paper as follows: In the second section we introduce some notations, definitions and basic results which are helpful in proving the main results. In the third section we prove perturbations results assuming a set of conditions and in the fourth section we prove the perturbation results with a different set of assumptions.
\section{Notations and Preliminary results}
Throughout the article we consider infinite dimensional complex Hilbert spaces which will be denoted by $H, H_1,H_2$ etc . The inner product and the
induced norm are denoted  by  $\langle,\rangle$ and
$||.||$ respectively. Let $T$ be a linear operator with domain $D(T)$ , a subspace of $H_1$ and taking values in $H_2$. The  graph $G(T)$ of $T$ is defined by  $G(T):={\{(x,x):x\in D(T)}\}\subseteq H_1\times H_2$. If $G(T)$ is
closed, then $T$ is called a closed operator. If $D(T)$ is dense in $H_1$, then $T$ is called a densely defined operator. For such an operator there exists a unique linear operator $T^*:D(T^*)\rightarrow H_1$, where
\begin{equation*}
D(T^*):={\{y\in H_2: x\rightarrow \langle Tx,y\rangle \, \text{for all}\, x\in D(T)\,\text{is continuous}}\}
\end{equation*}
and $\langle Tx,y\rangle =\langle x,T^*y\rangle$ for all $x\in D(T)$ and $y\in D(T^*)$. This operator is called as the adjoint of $T$. Whether or not $T$ is closed, $T^*$ is closed. By the  closed graph Theorem \cite[page 306]{riesznagy}, an everywhere defined
closed operator is bounded.  Hence the domain of an unbounded closed operator is a proper subspace of a Hilbert space.

The space of all bounded operators between $H_1$ and $H_2$ is
denoted by $\mathcal B(H_1,H_2)$ and the class of all closed operators between $H_1$ and $H_2$ is denoted by $\mathcal C(H_1,H_2)$. We write $\mathcal B(H,H)=\mathcal B(H)$  and $\mathcal C(H,H)=\mathcal C(H)$.

If $T\in \mathcal C(H_1,H_2)$, then  the null space and the range space of $T$ are denoted by $N(T)$ and $R(T)$ respectively and the space $C(T):=D(T)\cap N(T)^\bot$ is called the carrier of $T$.
In fact, $D(T)=N(T)\oplus^\bot C(T)$ \cite[page 340]{ben}.

If $T\in \mathcal L(H_1,H_2)$ and $S\in \mathcal L(H_2,H_3)$. Then $D(ST)={\{x\in D(T):Tx\in D(S)}\}$.

If $S$ and $T$ are closed operators  with the property that $D(T)\subseteq D(S)$ and $Tx=Sx$ for all $x\in D(T)$, then $S$ is called the restriction of $T$ and $T$ is called an extension of $S$.
Let $T\in \mathcal C(H)$ be densely defined. Then $T$ is said to be
\begin{itemize}
 \item normal if $D(T)=D(T^*)$ and $T^*T=TT^*$
 \item symmetric if $T\subseteq T^*$
 \item self-adjoint if $T=T^*$
 \item positive if  $T=T^*$ and $\langle Tx,x\rangle \geq 0$ for all $x\in D(T)$.
\end{itemize}

If $M$ is a closed subspace of a Hilbert space $H$, then $P_M$ is the orthogonal projection $P:H\rightarrow H$ with range $M$ and $S_M$ denote the unit sphere of $M$.
\begin{prop}
Let $T\in \mathcal C(H_1,H_2)$ be densely defined. Then the following statements hold:
\begin{enumerate}
\item $N(T)=R(T^*)^{\bot}$
\item $N(T^*)=R(T)^{\bot}$.
\end{enumerate}
\end{prop}

\begin{defn}\cite{ben, goldberg}\label{minmmodulus}
 Let $T\in \mathcal C(H_1,H_2)$ be densely defined. Then the quantity
  $\gamma(T):=\inf{\{\|Tx\|:x\in S_{C(T)}}\}$, is called the reduced minimum modulus of $T$.
\end{defn}

 \begin{defn}\label{geninv}(Moore-Penrose Inverse)\cite[Pages 314, 318-320]{ben}
Let $T\in \mathcal C(H_1,H_2)$ be densely defined. Then there exists
a unique densely defined operator $T^\dagger \in \mathcal
C(H_2,H_1)$ with domain $D(T^\dagger)=R(T)\oplus ^\bot R(T)^\bot$
and has the following properties:
\begin{enumerate}
\item $TT^\dagger y=P_{\overline{R(T)}}~y, ~\text{for all}~y\in D(T^\dagger)$.

\item $T^\dagger Tx=P_{N(T)^\bot} ~x, ~\text{for all}~x\in D(T)$.

\item $N(T^\dagger)=R(T)^\bot$.
\end{enumerate}
This unique operator $T^\dagger$ is called the \textit{Moore-Penrose inverse} of $T$.\\
The following property of $T^\dagger$ is also well known.
For every $y\in D(T^\dagger)$, let $$L(y):=\Big\{x\in D(T):
||Tx-y||\leq ||Tu-y||\quad \text{for all} \quad u\in D(T)\Big\}.$$
 Here any $u\in L(y)$ is called a \textit{least square solution} of the operator equation $Tx=y$. The vector  $x=T^\dagger y\in L(y),\,||T^\dagger y||\leq ||x||\quad \text{for all} \quad x\in L(y)$
 and it is called the  \textit{least square solution of minimal norm}.
 A different treatment of $T^\dagger$ is given in \cite[Pages 336, 339, 341]{ben},
 where it is called ``\textit{the Maximal Tseng generalized Inverse}".
\end{defn}
\begin{thm}\cite[Page 341]{ben}
Let $T\in \mathcal C(H_1,H_2)$ be densely defined. Then
\begin{enumerate}
\item $D(T^\dagger)=R(T)\oplus^\bot R(T)^\bot, \;\;
N(T^\dagger)=R(T)^\bot=N(T^*)$
\item $R(T^\dagger)=C(T)$
\item $T^\dagger$ is densely defined and $ T^\dagger \in \mathcal C(H_2,H_1)$
\item $T^{\dagger \dagger}=T$
\item $T^{* \dagger}=T^{\dagger *}$
\item $N(T^{* \dagger})=N(T)$
\item $T^*T$ and $T^\dagger T^{* \dagger}$ are positive and $(T^*T)^\dagger =T^\dagger T^{*
\dagger}$
\item $TT^*$ and $ T^{* \dagger}T^\dagger$ are positive and $(TT^*)^\dagger= T^{*
\dagger}T^\dagger$.
\end{enumerate}
\end{thm}
\begin{prop}\label{equivalentclosedrange}\cite{ben,kato}
Let $T\in \mathcal C(H_1,H_2)$ be densely defined. Then the following statements are equivalent;
\begin{enumerate}
\item $R(T)$ is closed
\item $R(T^*)$ is closed
\item $T_0:=T|_{C(T)}$ has a bounded inverse
\item there exists a $k>0$ such that $\|Tx\|\geq k\|x\|$ for all $x\in C(T)$
\item $T^{\dagger}$ is bounded. In fact, $\|T^{\dagger}\|=\frac{1}{\gamma(T)}$
\item $R(T^*T)$ is closed
\item $R(TT^*)$ is closed.
\end{enumerate}
\end{prop}
\begin{prop}\label{productadjoint}\cite[page 300]{riesznagy}
Let $T_1\in \mathcal C(H_1,H_3), ,T_2\in \mathcal C(H_1,H_2)$ be densely defined such that $D(T_1T_2)$ is dense in $H_2$. Then
\begin{enumerate}
\item $T_2^*T_1^*\subseteq (T_1T_2)^*$
\item If $T_1$ is bounded, Then $(T_1T_2)^*=T_2^*T_1^*$
\end{enumerate}
\end{prop}
\begin{prop}\cite{riesznagy}\label{productsumrelation}
If $T_i\in \mathcal C(H_1,H_2)\;(i=1,2)$ and $S\in \mathcal L(H_2,H_3)$, then $VT_1+VT_2\subseteq V(T_1+T_2)$.
\end{prop}
\begin{thm}\cite{schock,gr92,ped,riesznagy,rud}\label{closurerelations}
Let $T\in \mathcal C(H_1,H_2)$ be densely defined. Then
\begin{enumerate}
\item $T^*T$ is positive and $\overline{G(T|_{D(T^*T)})}=G(T)$
\item $(I+TT^*)^{-1}\in \mathcal B(H_1)$ and $||(I+T^*T)^{-1}||\leq
1$
\item $T(I+T^*T)^{-1}\in \mathcal B(H_1)$ and $(I+TT^*)^{-1}T\subset
T(I+T^*T)^{-1}$
\item $||T(I+T^*T)^{-1}||\leq \displaystyle \frac{1}{2}$
\item $\Big(T(I+T^*T)^{-1} \Big)^*=T^*(I+TT^*)^{-1}$
\item $\Big(T^*(I+TT^*)^{-1} \Big)^*=T(I+T^*T)^{-1}$.
\end{enumerate}
\end{thm}
\section{Perturbation of Moore-Penrose inverses-I}


\begin{prop}\label{sumclosedrange1}
Let $T\in \mathcal C(H_1,H_2)$ be densely defined and have a closed
range. Let $S\in \mathcal B(H_1,H_2)$. Then
\begin{enumerate}
\item $ST^\dagger T=S|_{D(T)} \Leftrightarrow  N(T)\subseteq N(T+S) \Leftrightarrow N(T)\subseteq N(S)$
\item If $||ST^\dagger ||<1$ and $ST^\dagger T=S|_{D(T)}$, then $N(T+S)=N(T)$.
\end{enumerate}
\end{prop}
\begin{proof}
It is easy to see that $N(T)\subseteq N(T+S)\Leftrightarrow
N(T)\subseteq N(S)$.

Suppose $ST^\dagger T=S|_{D(T)}$. If  $x\in N(T)$, then
$Sx=ST^\dagger Tx=0$. Hence $N(T)\subseteq N(S)$.

Now, for the  converse part, let us assume that $N(T)\subseteq
N(S)$.  Let $x\in D(T)$. Then $x=u+v$, where $u\in N(T), v\in C(T)$.
Then
\begin{align*}
ST^\dagger Tx = ST^\dagger Tv =SP_{\overline{R(T^\dagger)}}v &=SP_{\overline{C(T)}}v \\
              &=Sv \\
              &=S(u+v), \;(\text{ since}\; N(T)\subseteq N(S)) \\
              &=Sx.
\end{align*}

To prove $(2)$, in view of  $(1)$, it is enough to prove that
$N(T+S)\subseteq N(T)$. If  $x\in N(T+S)$,  by our assumption, we
have $Tx+ST^\dagger Tx=0$. That is $(I+ST^\dagger)Tx=0$. Since
$(I+ST^\dagger )^{-1}\in \mathcal B(H_1)$, we get a conclusion that
$Tx=0$ and hence $x\in N(T)$.
\end{proof}
\begin{thm}\label{sumclosedrange3}
Let $T\in \mathcal C(H_1,H_2)$ be densely defined operator with a
closed range and $S\in \mathcal B(H_1,H_2)$ be such that
\begin{itemize}
\item[(a)]
$||S||<\displaystyle \frac{1}{||T^\dagger||}$ and
\item[(b)] $N(T)\subseteq N(T+S)$.
\end{itemize}
Then $R(T+S)$ is closed.
\end{thm}
\begin{proof}
Since $R(T)$ is closed, we have $||Tx||\geq \gamma(T)\; ||x||$ for
each  $x\in C(T)$.  Since $C(T+S)\subseteq C(T)$, by the triangle inequality, for all $x\in
C(T+S)$,
 \begin{align*}
||Tx+Sx||&\geq \left|||Tx||-||Sx||\right|\\
         &\geq \left|\gamma(T)\;||x||-||S||\,||x||\right| \\
         &\geq (\gamma(T)-||S||)\,||x||.
\end{align*}
By Proposition \ref{equivalentclosedrange} and our assumption $\gamma(T)-||S||>0$.
 Again  by Proposition \ref{equivalentclosedrange}, $R(T+S)$ is closed.
\end{proof}
\begin{prop}\label{sumclosedrange2}
Let $T\in \mathcal C(H_1,H_2)$ be densely defined and have a closed
range. Let $S\in \mathcal B(H_1,H_2)$. Then
\begin{enumerate}
\item $TT^\dagger S=S $ \quad  if and only if \quad $ R(S)\subseteq R(T)$
\item If $TT^\dagger S=S$, then  $R(T+S)\subseteq R(T)$
\item  If $||T^\dagger S ||<1$ \quad and  \quad $TT^\dagger S=S$, then $R(T+S)=R(T)$.
\end{enumerate}
\end{prop}

\begin{proof}
 If $TT^\dagger S=S $, then it is obvious that $R(S)\subseteq R(T)$.

 On the other hand, if $R(S)\subseteq R(T)$, we have $TT^\dagger S =P_{R(T)}S=S$.

To prove $(2)$, let us assume that $TT^\dagger S=S$. Then by Proposition \ref{productsumrelation}, $$T+S= T+TT^\dagger S \subseteq T(I+T^\dagger
S).$$ That is $R(T+S)\subseteq R(T)$.

One part of  proof of $(3)$, namely, $R(T+S)\subseteq R(T)$ follows
from $(2)$.
 The condition $||T^\dagger S ||<1$ implies that $(I+T^\dagger S)^{-1}\in \mathcal B(H_1)$.
 Hence, if $y=Tx$ for some $x\in D(T)$, then by the surjectivity of $I+T^\dagger S$, there exists a $u\in D(T)$ such that  $x=(I+T^\dagger S)u$ .
 This  shows that  $y=T(I+T^\dagger S)u=Tu+Su \in R(T+S)$.
\end{proof}
\begin{thm}
Let $T\in \mathcal C(H_1,H_2)$ be densely defined and  have a closed
range and $S\in \mathcal B(H_1,H_2)$ be  such that
\begin{itemize}
\item[1.] $||T^\dagger S||<1$ \quad and
\item [2.]$TT^\dagger S=S$.
\end{itemize}
Then $R(T+S)$ is closed.
\end{thm}
\begin{proof}
By Proposition \ref{sumclosedrange2}, $R(T+S)=R(T)$. Since $R(T)$ is
closed so is  $R(T+S)$.
\end{proof}

An analogue of the following Theorem for the case of rectangular
matrices was proved by Stewart \cite{stewart}. The same result for
bounded operators on Banach spaces with a slight different
assumption was obtained by Nashed \cite{nashed}.
\begin{thm}\label{geninvofsum}
Let $T\in \mathcal C(H_1,H_2)$ be densely defined operator with a
closed range. Suppose  $S\in \mathcal B(H_1,H_2)$ satisfies the
following conditions:

\begin{enumerate}\label{newdagger}
\item $||T^\dagger S||<1$
 \item $TT^\dagger S=S$  and
 \item $ST^\dagger T=S|_{D(T)}$.
 \end{enumerate}
 Then \begin{itemize}
\item[(a)] $R(T+S)$ is closed and
\item[(b)]$(T+S)^\dagger=(I+T^\dagger S)^{-1}T^\dagger =T^\dagger
(I+ST^\dagger)^{-1}$. Hence $$T^\dagger =(T+S)^\dagger
(I+ST^\dagger).$$
 \end{itemize}
\end{thm}
\begin{proof}
First note that $T+S$ is a closed operator with $D(T+S)=D(T)$.
 The proof of $(a)$  follows by showing $R(T+S)=R(T)$, which is
proved in Proposition \ref{sumclosedrange2}. It remains to deduce
the formula for $T^\dagger$.

We want to prove  $T^\dagger=(I+T^\dagger S)(T+S)^\dagger$. It
suffices to prove $(T+S)^\dagger=(I+T^\dagger S)^{-1}T^\dagger$. Let
$U:=(I+T^\dagger S)^{-1}T^\dagger$. We show that $U$ satisfies all
the axioms of definition of the Moore-Penrose  inverse.

Note that the condition $ST^\dagger T=S|_{D(T)}$ together with
$||T^\dagger S||<1$, implies that $N(T)=N(T+S)$.

Let $z\in R(U)$. That is there exists a $y\in D(T^\dagger)$ such
that $z=(I+T^\dagger S)^{-1}T^\dagger y$. Hence $T^\dagger
y=z+T^\dagger Sz$. So $z=T^\dagger y-T^\dagger Sz \in
C(T)=C(T+S)\subseteq D(T+S)$

Now for $z\in D(T)$,
\begin{align*}
U(T+S)z&=(I+T^\dagger S)^{-1}T^\dagger (T+S)z\\
      &=(I+T^\dagger S)^{-1}T^\dagger (T+TT^\dagger S)z\\
      &=(I+T^\dagger S)^{-1}T^\dagger T(I+T^\dagger S )z\\
      &=(I+T^\dagger S)^{-1}P_{N(T)^{\bot}}(I+T^\dagger S )z \\
      &=(I+T^\dagger S)^{-1}(I+T^\dagger S )z\\
      &=z=P_{R(U)}z.
\end{align*}
In the fourth line of the above equations we have used the fact that
$T^\dagger T(I+T^\dagger S )=(I+T^\dagger S )T^\dagger T$ on $D(T)$,
which can be verified easily.

Now
\begin{align*}
(T+S)U=&(T+S)(I+T^\dagger S)^{-1}T^\dagger \\
      =&(T+TT^\dagger S)(I+T^\dagger S)^{-1}T^\dagger \quad (\text{since} \; S=TT^\dagger S)\\
      =&T(I+T^\dagger S)(I+T^\dagger S)^{-1}T^\dagger \\
      =&TT^\dagger \\
      =&P_{R(T)}\\
      =&P_{R(T+S)}. \qedhere
\end{align*}

Next, we show that $R(U)=R(S+T)^{\bot}$. Since $(I+T^\dagger S)^{-1}$ is invertible, $N(U)=N(T^{\dagger})=R(T)^{\bot}$. But by Proposition \ref{sumclosedrange2}, we have $R(T)=R(T+S)$.

The uniqueness of $(T+S)^\dagger$ follows  from Definition \ref{geninv}.

Since $R(S)\subseteq R(T)$, by Neumann series, we have
\begin{equation*}
 (I+ST^\dagger)^{-1}T^\dagger =\sum_{n=0}^\infty (-T^\dagger S)^{n}T^\dagger =\sum_{n=0}^\infty T^\dagger (-ST^\dagger)^n =T^\dagger (I+ST^\dagger).
\end{equation*}
\end{proof}

\begin{rmk}
 Proposition (\ref{sumclosedrange1})   and Proposition (\ref{sumclosedrange2}) tell us how to choose an operator $S$  satisfying the hypotheses of Theorem  (\ref{geninvofsum}).
 Let $\alpha $ be such that $0<\alpha <\displaystyle
 \frac{2}{||T^\dagger||}$. Let  $S_{\alpha}:=\alpha T(I+T^*T)^{-1}$ . Then by
 Proposition \ref{closurerelations}, \\$||T(I+T^*T)^{-1}||\leq \displaystyle \frac{1}{2}$. Thus  $||S_{\alpha}||\leq \displaystyle
 \frac{\alpha}{2}<\displaystyle \frac{1}{||T^\dagger||}$.

 Let $x\in D(T)$. Then
 \begin{align*}
 S_{\alpha} T^\dagger Tx&=\alpha T(I+T^*T)^{-1}T^\dagger Tx \\ &=\alpha (I+TT^*)^{-1}TT^\dagger Tx\\
                                                         &=\alpha (I+TT^*)^{-1}Tx \quad (\text{since}\;\; TT^\dagger T=T\quad \text{on}\;\; D(T))\\
                                                         &=\alpha \;T(I+T^*T)^{-1}x \quad \Big(\text{since}\;\; T(I+T^*T)^{-1}= (I+TT^*)^{-1}T \quad \text{on}\;\;D(T)\Big)\\
                                                         &=S_{\alpha}x.
 \end{align*}

 Now let $u\in H_1$. Then
\begin{align*}
 TT^\dagger S_{\alpha}u&=TT^\dagger \alpha T(I+T^*T)^{-1}u \\ &=\alpha T(I+T^*T)^{-1}u \quad (\text{since}\;\; TT^\dagger T=T\;\; \text{on}\;\; D(T))\\
                                                         &=\alpha T(I+T^*T)^{-1}u \\
                                                         &=S_{\alpha}u.
\end{align*}

Finally we have to verify that $||S_{\alpha}T^\dagger||<1$. As
$||S_{\alpha}||<\displaystyle \frac{\alpha}{2}$, we have
$$||S_{\alpha}T^\dagger||\leq ||S_{\alpha}||\;||T^\dagger||<1.$$
Hence the operators  $S_{\alpha} \quad (0<\alpha <\displaystyle
\frac{2}{||T^\dagger ||})$ satisfy Hypotheses of Theorem
\ref{geninvofsum}.
\end{rmk}
\begin{rmk}
From Theorem \ref{geninvofsum}, the following bounds for
$||(T+S)^\dagger-T^\dagger||$ can be deduced.\\
\begin{align*}
(T+S)^\dagger -T^\dagger&=(I+ST^\dagger)^{-1}T^\dagger-(I+ST^\dagger)(I+ST^\dagger)^{-1}T\dagger \\
                &=[I-(I+ST^\dagger-)](I+ST^\dagger)^{-1}T^\dagger\\
                &= (-ST^\dagger)(I+ST^\dagger)^{-1}T^\dagger .
\end{align*}
Hence
 $${||(T+S)^\dagger -T^\dagger||\leq \displaystyle \frac{||S||\,||T^\dagger||^2}{1-||T^\dagger S||}}.$$
\end{rmk}
\begin{rmk}\label{continuityofgamma}
Assume that $T$ and $S$ satisfy the conditions of  Theorem
\ref{newdagger}. We can deduce the following  formula.
\begin{align*}
|\gamma(S+T)-\gamma(T)|&=\left|\displaystyle \frac{1}{||(S+T)^\dagger||}-\displaystyle \frac{1}{||T^\dagger||}\right|\\
&=\left| \displaystyle \frac{||T^\dagger||-||(T+S)^\dagger||}{||T^\dagger||||(T+S)^\dagger||}\right| \\
&\leq \left| \displaystyle \frac{||T^\dagger||-||(T+S)^\dagger||}{||T^\dagger||||(T+S)^\dagger||}\right|\\
                &\leq \displaystyle \frac{||S||\,||T^\dagger||^2}{||T^\dagger||\,||(T+S)^\dagger||\,||[1-||T^\dagger S]||}\\
                &=\beta ||S||
\end{align*}
where $\beta=\displaystyle
\frac{||T^\dagger||}{||(T+S)^\dagger||[1-||T^\dagger S||]}>0.$
\end{rmk}
\begin{cor}
Let $T\in \mathcal C(H_1,H_2)$ be densely defined and have a closed
range. Assume that $S_n\in \mathcal B(H_1,H_2)$  satisfy all the
conditions of Theorem \ref{geninvofsum} and $S_n\rightarrow 0$. Then
\begin{enumerate}
\item $(T+S_n)^\dagger \rightarrow T^\dagger$ in the operator
norm of $\mathcal B(H_2,H_1)$.
\item $\gamma(T+S_n)\rightarrow \gamma(T)$
as $n\rightarrow \infty$.
\end{enumerate}
\end{cor}
\begin{proof}
The proofs follow from  Theorem \ref{newdagger} and  Remark
\ref{continuityofgamma}.
\end{proof}
The following results which are proved in \cite{dinghuang94}, can be
obtained as particular cases of the above results.
\begin{cor}\cite[Lemma 3.3]{dinghuang94}
Let $T\in \mathcal B(H_1,H_2)$ be injective with closed range, and
let $S\in \mathcal B(H_1,H_2)$ be such that
\begin{itemize}
 \item[(a)] $R(S)\subseteq R(T)$
\item [(b)]$||T^\dagger S||<1$.
\end{itemize}

Then
\begin{itemize}
\item[(1)] $T+S$ is injective
\item [(2)]$R(T+S)=R(T)$
\item [(3)]$(T+S)^\dagger =(I+T^\dagger S)^{-1}T^\dagger =T^\dagger (I+ST^\dagger)^{-1}$
\item [(4)]$||(T+S)^\dagger||\leq \displaystyle \frac{||T^\dagger||}{1-||T^\dagger S||}$
\item [(5)]$||(T+S)^\dagger -T^\dagger ||\leq \displaystyle \frac{||T^\dagger S||||T^\dagger||}{1-||T^\dagger S||}$.
\end{itemize}
\end{cor}
\begin{proof}
 By Proposition (\ref{sumclosedrange1}), $(a)$ is equivalent to the condition $TT^\dagger S=S$. Since $T$ is injective, $T^\dagger T=I|_{D(T)}$.
 Hence $ST^\dagger T=SI|_{D(T)}=S|_{D(T)}$. Now by Theorem (\ref{geninvofsum}),  $(T+S)^\dagger =(I+T^\dagger S)^{-1}T^\dagger =T^\dagger (I+ST^\dagger)^{-1}$.
 The other relations  can be proved by using the Neumann series.
\end{proof}

\begin{cor}\cite[Lemma 3.4]{dinghuang94}
Let $T\in \mathcal B(H_1,H_2)$ be surjective and $S\in \mathcal
B(H_1,H_2)$ be such that
\begin{enumerate}
 \item[(a)] $N(T)\subseteq N(S)$
\item [(b)]$||ST^\dagger||<1$.
\end{enumerate}
Then
\begin{enumerate}
 \item $T+S$ is surjective
\item $N(T+S)=N(T)$
\item $(T+S)^\dagger =T^\dagger (I+ST^\dagger)^{-1}=(I+T^\dagger S)^{-1}T^\dagger$
\item $||(T+S)^\dagger||\leq \displaystyle \frac{||T^\dagger||}{1-||ST^\dagger ||}$
\item $||(T+S)^\dagger -T^\dagger ||\leq \displaystyle \frac{||ST^\dagger ||||T^\dagger||}{1-||ST^\dagger ||}$.
\end{enumerate}
\end{cor}
\begin{proof}
 By Proposition (\ref{sumclosedrange1}), $N(T)\subseteq N(S) \Leftrightarrow ST^\dagger T=S|_{D(T)}$. Since $T$ is surjective,  $TT^\dagger =I$, that is $STT^\dagger =S$.
 Now applying Theorem (\ref{geninvofsum}), we get the expression for $(S+T)^\dagger$. The bounds for $||(S+T)^\dagger ||$ and $||(S+T)^\dagger -T^\dagger ||$ follows from Neumann series and Theorem \ref{geninvofsum}.
\end{proof}


\begin{cor}\cite[Lemma 4.3]{dinghuang94}
Let $T\in \mathcal B(H_1,H_2)$ have a closed range and $S\in
\mathcal B(H_1,H_2)$ be such that
\begin{enumerate}
\item $N(T)\subseteq N(S)$
\item $||S||||T^\dagger||<1$.
\end{enumerate}
Then
\begin{enumerate}
 \item[(i)] $N(T+S)=N(T)$
\item [(ii)]$||(T+S)^\dagger ||\leq  \displaystyle \frac{||T^\dagger||}{1-||S||||T^\dagger||}$.
\end{enumerate}
\end{cor}
\begin{proof}
 The proof of $(i)$, follows by Proposition (\ref{sumclosedrange1}) and Theorem
 (\ref{sumclosedrange3}).

\underline{Proof of $(ii)$}:

By Theorem \ref{sumclosedrange3}, we have  $$||Tx+Sx||\geq
(\gamma(T)-||S||)||x||\quad \text{for all}\quad x\in C(T+S).$$ Hence
$\gamma(T+S)\geq \gamma(T)-||S||$. That is
$$
 \displaystyle \frac{1}{||(T+S)^\dagger||} \geq \displaystyle \frac{1}{||T^\dagger||}-||S||
                                           \geq \displaystyle \frac{1-||S||||T^\dagger||}{||T^\dagger||}.$$
Therefore $||(T+S)^\dagger||\leq \displaystyle
\frac{||T^\dagger||}{1-||S||||T^\dagger||}.$
\end{proof}

\section{Perturbation of Moore-Penrose inverses-II}
In this section we prove perturbation results for Moore-Penrose inverses under different assumptions. First we prove a representation result for the Moore-Penrose inverse.
\begin{thm}\label{repn}
Let $T\in \mathcal C(H_1,H_2)$ be densely defined and have closed range. Then
\begin{equation*}
(T^*T)^{\dagger}T^*\subseteq T^*(TT^*)^{\dagger}=T^{\dagger}.
\end{equation*}
\end{thm}
\begin{proof}
First we prove that $(T^*T)^*\subseteq T^*(TT^*)^{\dagger}$. Since $R(TT^*)$ is closed, $(TT^*)^{\dagger}\in \mathcal B(H_1)$. Hence $D((T^*T)^{\dagger}T^*)=D(T^*)$. Next, since $R(TT^*)^{\dagger}=C(TT^*)\subseteq D(TT^*)\subseteq D(T^*)$, by definition of product of operators, we have that $D(T^*(TT^*)^{\dagger})=H_2$ and $T^*T(TT^*)^{\dagger}$ being a composition of a closed operator with a bounded operator, is closed. Hence by the closed graph theorem \cite{riesznagy}, $T^*(TT^*)^{\dagger}$ is bounded. This shows that $D\big((T^*T)^{\dagger}T^*\big)\subseteq D\big(T^*(TT^*)^{\dagger}\big)$.
Now, let $z\in D\big((T^*T)^{\dagger}T^*\big)$. Then $z=u+Tv$, where $u\in N(T^*)=N(T^{\dagger})=N\big((TT^*)^{\dagger}\big)$ and $v\in C(TT^*)$. Hence
\begin{equation*}
(T^*T)^{\dagger}T^*z=(T^*T)^{\dagger}T^*(u+Tv)=(T^*T)^{\dagger}T^*Tv=P_{N(TT^*)^{\bot}v}=v.
\end{equation*}
Similarly, for $z\in D\big((T^*T)^{\dagger}T^*\big)$, we have
\begin{align*}
T^*(TT^*)^{\dagger}z=T^*T{T^*}^{\dagger}z=T^*{T^*}\dagger T^{\dagger}(u+Tv)&=T^*{T^*}^{\dagger}T^{\dagger}Tv\\
                                                                            &=P_{R(T^*)^{\bot}P_{N(T)^{\bot}}}v\\
                                                                            &=v.
\end{align*}
The result, $T^*(TT^*)^{\dagger}=T^{\dagger}$ is proved in \cite[Corollary 3.3]{cwginclusions}.
\end{proof}
\begin{rmk}
If $T$ is a bounded operator with closed range, then $(T^*T)^{\dagger}T^*=T^*(TT^*)^{\dagger}=T^{\dagger}$.
\end{rmk}
Now we obtain the riverse order law for Moore-Penrose inverse assuming particular conditions. The same result for bounded operators can be found in \cite{rbholmes77}.
\begin{thm}\label{riverseorderlaw}
Let $F\in \mathcal B(H_2,H_3)$ with $R(F^*)=H_2$ and $G\in \mathcal C(H_1,H_2)$ be densely defined such that $R(G)=H_2$ and $A=FG$ is closed. Then
\begin{equation*}
A^{\dagger}=G^{\dagger}F^{\dagger}=G^*(GG^*)^{-1}(F^*F)^{-1}F^*.
\end{equation*}
\end{thm}
\begin{proof}
Here $D(A)=D(G)$, which is dense in $H_1$. Hence $A^*$ exists and equals to $G^*F^*$ by Proposition \ref{productadjoint}. Using the fact that $G$ is onto we can prove that $R(A)=R(F)$ and hence by Proposion \ref{equivalentclosedrange}, $R(A)$ is closed, which is equivalent to saying that $A^{\dagger}$ is bounded. Since $F$ is one-to-one, $F^{\dagger}F=I$ and $(F^*F)^{-1}\in \mathcal B(H_1)$. Also, since $G$ is onto, $GG^{\dagger}=I$ and $(GG^*)^{-1}\in \mathcal B(H_2)$. Thus we have
\begin{equation*}
A^{\dagger}=A^{\dagger}AA^{\dagger}=A^{\dagger}FGA^{\dagger}=(A^{\dagger}F)(GA^{\dagger}).
\end{equation*}
In view of Theorem \ref{repn}, it is enough to show that
\begin{enumerate}
\item \label{1stpart} $A^{\dagger}F=G^*(GG^*)^{-1}$ and
\item \label{2ndpart} $GA^{\dagger}=(F^*F)^{-1}F^*$.
\end{enumerate}
First we prove (\ref{1stpart}. For this, consider
\begin{equation}\label{repneqn1}
G^*F^*{A^{\dagger}}^*=(FG)^{*}{A^{\dagger}}^*=(A^{\dagger}A)^{*}=A^{\dagger}A.
\end{equation}
Hence by Equation \ref{repneqn1}, we have that
\begin{equation}\label{repneq2}
FGG^*F^*{A^{\dagger}}^*=FGA^{\dagger}A=AA^{\dagger}A=A=FG.
\end{equation}
pre multiplying Equation \ref{repneq2} by $F^{\dagger}$, we get
\begin{equation}\label{repneq3}
F^{\dagger}FGG^*F^*{A^{\dagger}}^*=F^{\dagger}FG.
\end{equation}
Since, $FF^{\dagger}=I$, The Equation (\ref{repneq3}), becomes $GG^*F^*{A^\dagger}^*=G$. That is $(A^{\dagger}F)^*=(GG^*)^{-1}G$. Taking adjoint both sides, we can conclude
that $A^{\dagger}F=G^*(GG^*)^{-1}$.

To prove, $GA^{\dagger}=(F^*F)^{-1}F^*$, note that
\begin{equation*}
AA^{\dagger}=(AA^{\dagger})^*={A^{\dagger}}^*A^*={A^{\dagger}}^*G^*F^*.
\end{equation*}
Also \begin{equation}\label{repneq4}
FG=A=AA^{\dagger}A={A^\dagger}^*G^*F^*FG.
\end{equation}
Post multiplying Equation (\ref{repneq4}) by $G^{\dagger}$ and noting that $GG^{\dagger}=I$, we get that $F={A^\dagger}^*G^*F^*F=(GA^{\dagger})^*F^*F$. Therefore, $(GA^{\dagger})^*=F(FF^*)^{-1}$, taking adjoint both sides, we get that $GA^{\dagger}=(F^*F)^{-1}F^*$.
\end{proof}
\begin{lemma}\label{relativeneumanninequality}\cite[lemma 2.2]{dingj}
Let $T\in \mathcal B(H)$ be such that
\begin{equation*}
\|Tx\|\leq \lambda_1\|x\|+\lambda_2\|(I+T)x\| \; \text{for all}\; x\in H,
\end{equation*}
where $\lambda_j<1,\; j=1,2$. Then $\lambda_j\in (-1,1)$ and $I+T$ is bijective.

Moreover,
\begin{align*}
\dfrac{1-\lambda_1}{1+\lambda_2}\|x\|&\leq \|(I+T)x\|\leq \dfrac{1+\lambda_1}{1-\lambda_2}\|x\|\; \text{for all}\; x\in H,\\
\dfrac{1-\lambda_2}{1+\lambda_1}\|y\|&\leq \|(I+T)^{-1}y\|\leq \dfrac{1+\lambda_2}{1-\lambda_1}\|y\|\; \text{for all}\; y\in H.
\end{align*}
\end{lemma}
\section{Perturbation results: $S$-bounded operators}
\begin{thm}\label{mainperturb}
Let $T\in \mathcal C(H_1,H_2)$ be densely defined and $S\in \mathcal B(H_1,H_2)$  be such that
\begin{equation}\label{relativeboundeq}
\|Sx\|\leq \lambda_1\|Tx\|+\lambda_2\|(S+T)x\| \text{for all}\; x\in D(T),
\end{equation}
where $\lambda_1<1$ and $\lambda_2\in \mathbb R$. Then $\lambda_2>-1$ and
\begin{enumerate}
\item\label{perturbclosedrange} $\|(S+T)x\|\geq \frac{1-\lambda_1}{1+\lambda_2}\|Tx\|\; \text{for all}\; x\in D(T)$
\item\label{equalityofnullsp} $N(T)=N(T+S)=N(T)\cap N(S)$
\item \label{sumclosedrange}If $R(T)$ is closed, then $R(T+S)$ is closed. In this case,
\begin{equation*}
\|(T+S)^{\dagger}\|\leq \dfrac{1+\lambda_2}{1-\lambda_1}\|T^{\dagger}\|
\end{equation*}
\item\label{repnofperturbmpi} If in addition, $T$ is onto, the $T+S$ is also onto and
\begin{equation*}
(T+S)^{\dagger}=T^{\dagger}(I+ST^{\dagger})^{-1}.
\end{equation*}
\end{enumerate}
\end{thm}
\begin{proof}
Proof of (\ref{perturbclosedrange}): By the triangle inequality we have that
\begin{align*}
\|Sx+Tx\|&\geq \big|\|Tx\|-\|Sx\|\big|\\
         &\geq \big |\|Tx\|-\lambda_1\|Tx\|-\lambda_2\|Sx+Tx\|\big|.
\end{align*}
Thus \begin{equation}\label{repneq5}
(1+\lambda_2)\|Sx+Tx\|\geq (1-\lambda_1)\|Tx\|.
\end{equation}
Since, $\lambda_1<1$, it follows that $\lambda_2>1$. Now, by Equation (\ref{repneq5}), we have
\begin{equation*}
\|Sx+Tx\|\geq \dfrac{1-\lambda_1}{1+\lambda_2}\|Tx\|, \; \text{for all}\; x\in D(T).
\end{equation*}
Proof of (\ref{equalityofnullsp}): By (\ref{perturbclosedrange}), it follows that  $N(T+S)\subseteq N(T)$. To prove the otherway, let $x\in N(T)$. Then By Equation (\ref{relativeboundeq}), we have $\|Sx\|\leq \lambda_2\|Sx\|$. That is $(1-\lambda_2)\|Sx\|=0$. As $\lambda_2>-1$, it follows that $\|Sx\|=0$. This implies that $x\in N(S)\cap N(T)$. Hence $N(T)=N(T+S)$.

Proof of (\ref{sumclosedrange}): By (\ref{perturbclosedrange}), we have that
\begin{equation}\label{repneq6}
\gamma(T+S)\leq \frac{1-\lambda_1}{1+\lambda_2}\gamma(T).
\end{equation}
If $R(T)$, then $\gamma(T)>0$. Hence $\gamma(T+S)>0$. By \ref{equivalentclosedrange}, $R(T+S)$ is closed. Using proposition and Equation (\ref{repneq6}), we get the required inequality.

Proof of (\ref{repnofperturbmpi}): Since $N(T+S)=N(S)$, it follows that $ST^{\dagger}T=S|_{D(T)}$ by Proposition \ref{sumclosedrange1}. Hence $S+T=ST^{\dagger}T+T=(I+ST^{\dagger})T$. If $T$ is onto, we have for every $y\in H_2$, $TT^{\dagger}y=y$. Hence by Hypothesis we have
\begin{align*}
\|ST^{\dagger}y\|&\leq \lambda_1\|TT^{\dagger}y\|+\lambda_2\|(S+T)T^{\dagger}y\|\\
                 &\leq \lambda_1\|y\|+\lambda_2\|(I+ST^{\dagger}i+T)y\|.
\end{align*}
So by Lemma \ref{relativeboundeq}, $(I+ST^{\dagger})$ is bijective and hence $(T+S)(H_1)=(I+ST^{\dagger})(TH_1)=(I+ST^{\dagger})(H_2)=H_2$. This shows that $T+S$ is onto.

It remains to show the representation. Note that $H_2=R(T+S)\subseteq R(I+ST^{\dagger})$, concluding $I+ST^{\dagger}$ is onto. Let $F=T$ and $G=I+ST^{\dagger}$. Then $F$ and $G$ satisfy the hypotheses of Theorem \ref{riverseorderlaw}. Hence by (\ref{riverseorderlaw}), we get that
\begin{align*}
(T+S)^{\dagger}(I+ST^{\dagger})&=\big(ST^{\dagger}T+T\big)^{\dagger}(I+ST^{\dagger})\\
                               &=\big((I+ST^{\dagger})T\big)^{\dagger}(I+ST^{\dagger})\\
                               &=T^{\dagger}(I+ST^{\dagger})^{-1}(I+ST^{\dagger})\\
                             &=T^{\dagger}.
\end{align*}
Therefore $(S+T)^{\dagger}=T^{\dagger}(I+ST^{\dagger})^{-1}$.
\end{proof}
We have the following consequence.
\begin{cor}
Let $T\in \mathcal C(H_1,H_2)$ be densely defined and $S\in \mathcal B(H_1,H_2)$ be such that
\begin{equation*}
\|(S-T)x\|\leq \lambda_1\|Tx\|+\lambda_2\|Sx\| \; \text{for all}\; x\in D(T),
\end{equation*}
where $\lambda_1<1$. Then $\lambda_2>-1$ and
\begin{enumerate}
\item $N(T)=N(S)$
\item $\|Sx\|\geq \frac{1-\lambda_1}{1+\lambda_2}\|Tx\|, \text{for all}\; x\in D(T)$
\item If $R(T)$ is closed, then $R(S)$ is closed. In this case,
\begin{equation*}
\|S^{\dagger}\|\leq \frac{1+\lambda_2}{1-\lambda_1}\|T^{\dagger}\|
\end{equation*}
\item In addition, if $T$ is onto, then $S$ is onto and
\begin{equation*}
S^{\dagger}=T^{\dagger}\big(I+(S-T)T^{\dagger}\big)^{-1}=T^{\dagger}\displaystyle \sum_{n=0}^\infty \big((S-T)T^{\dagger}\big)^n.
\end{equation*}
\end{enumerate}
\end{cor}
\begin{proof}
All these statements can be proved by substituting $S$ by $S-T$ in Theorem \ref{mainperturb}.
\end{proof}
As a particular case of Theorem \ref{mainperturb}, we can deduce the following error bound for the Moore-Penrose inverse.
\begin{cor}\label{consequenceofmain}
Let $T$ and $S$ be as in Theorem \ref{mainperturb}. In addition assume that $T$ is onto and $\lambda_2=0$. Then
\begin{enumerate}
\item\label{neumannsbound} $\|ST^{\dagger}\|<1$ and $(I+ST^{\dagger})^{-1}\in \mathcal B(H_2)$. In this case,
\begin{equation*}
\|(I+ST^{\dagger})^{-1}\|\leq \frac{1}{1-\|ST^{\dagger}\|}
\end{equation*}
\item\label{perturbofmpibound} $\|(S+T)^{\dagger}-T^{\dagger}\|\leq \frac{\|T^{\dagger}\|^2\|S\|}{1-\|ST^{\dagger}\|}$.
\end{enumerate}
\end{cor}
\begin{proof}
The proof of (\ref{neumannsbound}) follows from standard result on bounded operators.

Proof of (\ref{perturbofmpibound}): Consider

\begin{align*}
\|(S+T)^{\dagger}-T^{\dagger}\|&=\|T^{\dagger}(I+ST^{\dagger})^{-1}-T^{\dagger}(I+ST^{\dagger})^{-1}(I+ST^{\dagger})\|\\
                               &=\|T^{\dagger}(I+ST^{\dagger})^{-1}\big(I-(I-ST^{\dagger})\big)\|\\
                               &\leq \frac{\|T^{\dagger}\|^2\|S\|}{1-\|ST^{\dagger}\|}.\qedhere
\end{align*}
\end{proof}

\bibliographystyle{amsplain}
\bibliography{rameshmainthesis}
\end{document}